\documentclass{siamonline1116}
\usepackage{amsmath, amssymb, amsbsy, graphicx, caption}
\usepackage[export]{adjustbox}
\usepackage{subcaption}

\def\Rec{\mathcal R} 
 
\def\R{\mathbb{R}}

\def\N{\mathbb{N}}

\def\p{\partial}

\def\range{\textnormal{Range}}
\def\domain{\textnormal{Domain}}

\def\diam{\textnormal{diam}}
\def\linearSpan{\textnormal{span}}

\newcommand{\norm}[1]{\left\|#1 \right\|}

\def\Vol{\textnormal{Vol}_{g}}

\def\wavecap{\textnormal{cap}}

\title{Recovery of a Smooth Metric via Wave Field and Coordinate
  Transformation Reconstruction% 
  \thanks{
    Submitted to the editors October 6, 2017.  
    \funding{M.V.d.H. gratefully acknowledges support from the Simons
      Foundation under the MATH + X program, the National Science
      Foundation under grant DMS-1559587, and the corporate members of
      the Geo-Mathematical Imaging Group at Rice University. 
      P.K. was supported in part by the
      Geo-Mathematical Imaging group at Rice University.
      L.O. was supported by EPSRC grants EP/L026473/1 and
      EP/P01593X/1. }}}
\author{
Maarten V. de Hoop%
\thanks{Simons Chair in Computational and Applied Mathematics and
  Earth Science, Rice University, Houston TX 77005, USA
  (\email{mdehoop@rice.edu}).}%
 \and 
 Paul Kepley% 
 \thanks{Department of Mathematics, Purdue University, West Lafayette, IN 47907 
 (\email{pkepley@purdue.edu}).}%
 \and 
 Lauri Oksanen% 
 \thanks{Department of Mathematics, University College London, Gower Street, London WC1E 6BT, UK
 (\email{l.oksanen@ucl.ac.uk}).}%
}
\headers{M. V. DE HOOP, P. KEPLEY, AND L. OKSANEN}{RECOVERY OF A SMOOTH METRIC}

\ifpdf
\hypersetup{ pdftitle={Recovery of a smooth metric via wave field and coordinate
  transformation reconstruction} }
\fi

\begin{document}

%% make the title
\maketitle

\begin{abstract}
  In this paper, we study the inverse boundary value problem for the
  wave equation with a view towards an explicit reconstruction
  procedure.  We consider both the anisotropic problem where the
  unknown is a general Riemannian metric smoothly varying in a domain,
  and the isotropic problem where the metric is conformal to the
  Euclidean metric.  Our objective in both cases is to construct the
  metric, using either the Neumann-to-Dirichlet (N-to-D) map or
  Dirichlet-to-Neumann (D-to-N) map as the data.  In the anisotropic
  case we construct the metric in the boundary normal (or
  semi-geodesic) coordinates via reconstruction of the wave field in
  the interior of the domain.  In the isotropic case we can go further
  and construct the wave speed in the Euclidean coordinates via
  reconstruction of the coordinate transformation from the boundary
  normal coordinates to the Euclidean coordinates.  Both cases utilize
  a variant of the Boundary Control method, and work by probing the
  interior using special boundary sources.  We provide a computational
  experiment to demonstrate our procedure in the isotropic case with
  N-to-D data.
\end{abstract}



%\todo{Title changed, abstract needs to be changed}

\section{Introduction}

We study the inverse boundary value problem for the
wave equation from a computational point of view. Specifically, let $M \subset \R^n$ be a compact connected domain
with smooth boundary $\p M$, and let $c(x)$ be an unknown smooth 
strictly positive function on $M$. 
Let $u = u^f$ denote the solution to the
wave equation on $M$, with Neumann source $f$, 
\begin{equation}
  \label{wave_eq_iso}
  \begin{array}{rcl}
    \p_t^2 u - c^2(x)\Delta u &=& 0, \quad \textnormal{in $(0,\infty) \times M$}, \\
    \p_{\vec n} u|_{x \in \p M} &=& f,  \\
    u|_{t=0}  = \p_t u|_{t=0}, &=& 0.
  \end{array}
\end{equation}
Here $\vec n$ is the inward pointing (Euclidean) unit normal vector on $\p M$.
Let $T > 0$ and let $\Rec \subset \p M$ be open.
We suppose that the restriction of the 
Neumann-to-Dirichlet (N-to-D) map on $(0,2T) \times \Rec$ is known, and denote this map by
$\Lambda_{\Rec}^{2T}$. It is defined by
\begin{equation*}
  \Lambda_{\Rec}^{2T} : f \mapsto u^f|_{(0,2T) \times
    \Rec}, \quad f \in C_0^\infty((0,2T) \times \Rec).
\end{equation*}
The goal of the inverse boundary value problem is to use the data
$\Lambda_{\Rec}^{2T}$ to determine the wave speed $c$ in a subset $\Omega \subset M$ modelling the region of interest. 

Our approach to solve this inverse boundary value problem is based on the Boundary Control method that originates from \cite{Belishev1987}. 
There exists a large number of variants of the Boundary Control method in the theoretical literature, see e.g. the review \cite{Belishev2007}, the monograph \cite{Katchalov2001},
and the recent theoretical uniqueness 
\cite{Eskin2015,Kurylev2015} and stability results \cite{Bosi2017}.
We face an even wider array of possibilities when designing computational implementations of the method.
Previous computational studies of the method
include \cite{Belishev1999,Pestov2010}
and the recent work \cite{Belishev2016}.

Motivated by applications to seismic imaging, we are particularly interested in the problem with 
partial data, that is, the case $\Rec \ne \p M$.
All known variants of the Boundary Control method that work with partial data
require solving ill-posed control problems,
and this appears to form the bottleneck of the resolution of the method. 
In this paper we consider this issue
from two perspectives:
we show that the steps of the method, apart from solving the control problems, are stable;
and present a computational implementation of 
the method with a regularization for the control problems.

In addition to the above isotropic problem with the scalar speed of sound $c$, we consider an anisotropic problem
and a variation where the data is given by the Dirichlet-to-Neumann map 
rather than the Neumann-to-Dirichlet map,
see the definitions (\ref{wave_eq_aniso}) and (\ref{DtoN}) below.
We propose a computational method to reduce the anisotropic inverse boundary value problem to a problem with data in the interior of $M$.
Analogously to elliptic inverse problems with internal data \cite{Bal2013},
this hyperbolic internal data problem may be of independent interest,
and we show a Lipschitz stability result for the problem under a geometric assumption. 
We show the correctness of our method without additional geometric
assumptions (Proposition \ref{prop:computingG}), but for the stability
of the internal data problem in the anisotropic case we require
additional convexity condition to be satisfied (Theorem
\ref{th_main}).

Our computational approach in the isotropic case combines
two techniques that have been successfully used in the previous literature. 
To solve the ill-posed control problems, 
we use the regularized optimization approach that originates from \cite{Bingham2008}. This is combined with the use of the eikonal equation as in the previous computational studies \cite{Belishev1999,Belishev2016}.
The main difference between \cite{Belishev1999,Belishev2016} and the present work is that in \cite{Belishev1999,Belishev2016} the ill-posed control problems, and the subsequent reconstruction of internal information (see Section \ref{sec:interior} below), are 
implemented using the so-called wave bases rather than regularized optimization. Another distinction  is that we do not rely upon the amplitude formula from geometric optics to extract internal information. Instead, we use the boundary data to construct sources that allow us to extract localized averages of waves and harmonic functions in the interior.

Our motivation to study the Boundary Control method comes from potential applications in seismic imaging. The prospect is that the method could provide a good initial guess for the local optimization methods currently in use in seismic imaging.
These methods suffer from the fact that they may converge to a local minimum of the cost function and thus fail to give the true solution to the imaging problem \cite{Symes2009}. On the other hand, the Boundary Control method is theoretically guaranteed to converge to the true solution, however, in practice, we need to give up resolution in order to stabilize the method. 
The numerical examples in this paper show that, when regularized suitably, the method can stably reconstruct smooth variations in the wave speed. 

We reconstruct the wave speed only in a region near the measurement surface $\Rec$, since at least in theory, it is possible to iterate this procedure in a layer stripping fashion. 
The layer stripping alternates between the local reconstruction  step as discussed in this paper and the so-called redatuming step that propagates the measurement data through the region where the wave speed is already known. We have developed the redatuming step computationally in \cite{Hoop2016a}.

We will not attempt to give an overview of computational methods for 
coefficient determination problems for the wave equation
that are not based on the Boundary Control method.
However, we mention the interesting recent computational work 
\cite{Baudouin2016}
that is based on the so-called Bukhgeim-Klibanov method \cite{Bukhgeuim1981}.  
We note that the Bukhgeim-Klibanov method uses different data from the Boundary Control method, requiring only a single instance of boundary values, but that it also requires that the initial data are non-vanishing.
We mention also another reconstruction method 
that uses a single measurement \cite{Beilina2008,Beilina2012}.
This method is based on a reduction to a non-linear integro-differential equation, and there are several papers 
on how to solve this equation (or an approximate version of it), see \cite{Klibanov2015,Klibanov2017} for
recent results including computational implementations.
Finally, we mention \cite{Kabanikhin2005}
for a thorough comparison of several methods in the $1+1$-dimensional case.

\section{Notation and techniques from the Boundary Control method}

The Boundary Control (BC) method is based on the geometrical aspects of wave propagation. These are best described using the language of Riemannian geometry, and in that spirit we define the isotropic Riemannian metric $g = c(x)^{-2} dx^2$ associated to the wave speed $c(x)$ on $M$. 
Put differently, in the Cartesian coordinates of $M$, 
the metric tensor $g$ is represented by $c(x)^{-2}$ times the identity matrix. Now the distance function of the Riemmannian manifold $(M,g)$ encodes the travel times of waves between points in $M$,
and singular wave fronts propagate along the geodesics of $(M,g)$.

We will also discuss the case of an anisotropic wave speed
and the Dirichlet-to-Neumann (D-to-N) map. 
This means that $g$ is allowed to be an arbitrary smooth Riemannian metric on $M$, and we consider the wave equation
\begin{equation}
\label{wave_eq_aniso}
  \begin{array}{rcl}
    \p_t^2 u - \Delta_g u &=& 0, \quad \textnormal{in $(0,\infty) \times M$}, \\
    u|_{x \in \p M} &=& f,  \\
    u|_{t=0}  = \p_t u|_{t=0}, &=& 0
  \end{array}
\end{equation}
together with the map 
\begin{equation}
\label{DtoN}
  \Lambda_{\Rec}^{2T} : f \mapsto -\partial_\nu u^f|_{(0,2T) \times
    \Rec}, \quad f \in C_0^\infty((0,2T) \times \Rec).
\end{equation}
Here $\Delta_g$ is the Laplace-Beltrami operator on the Riemannian manifold $(M,g)$, and $\nu$ is the inward pointing unit normal vector to $\p M$ with respect to the metric $g$.
All the techniques in this section are the same for both the isotropic and anisotropic cases and for both the choices of data N-to-D and D-to-N. The negative sign is chosen in (\ref{DtoN}) 
to unify the below formula (\ref{Blago}) between the two choices of data. 
We leave it to the reader to adapt the formulations for the isotropic case with D-to-N and the anisotropic case with N-to-D.

%Note that the isotropic wave equation (\ref{wave_eq_iso})
%is not of the form (\ref{wave_eq_aniso}),
%and that we could cover both the cases by considering the weighted Laplace-Beltrami operators of the form 
%$\mu^{-1}
%\divergence \mu \nabla_g$
%where the smooth function $\mu$ is considered as a multiplication operator, see e.g. the discussion in \cite{Bingham2008} or \cite{Hoop2016}.
%As the methods we present are different in the isotropic and anisotropic cases, we prefer to avoid the notational burden that comes with the weight factor $\mu$.
%However, the techniques in this section are the same in both the cases.

The BC method is based upon
approximately solving control problems of the form,
\begin{equation}
\label{abstract_cp}
  \textnormal{find $f$ for which $u^f(T,\cdot) = \phi$}
\end{equation}
where the target function $\phi \in L^2(M)$ belongs to an appropriate
class of functions
so that the problem can be solved without knowing the wave speed. 
One could call this problem a {\em blind control problem}.
The earliest formulations of the BC method solved
such control problems by applying a Gram-Schmidt orthogonalization
procedure to the data. However, as noted in \cite{Bingham2008}, this
procedure may itself be ill-conditioned. As a result, regularization
techniques were introduced to the BC method \cite{Bingham2008}. One
issue that arises with this particular regularized approach to the BC
method is that there is no explicit way to choose the target function
$\phi$. Thus in \cite{Oksanen2011} a variation of the regularized approach was
introduced, where the target functions $\phi$ were restricted to the
set of characteristic functions of domains of influence. This
technique uses global boundary data (i.e. $\Rec = \p M$) to
construct boundary distance functions.
In \cite{Hoop2016} we introduced a modification of \cite{Oksanen2011}
that allowed us to localize the problem and work with partial boundary data
(i.e. $\Rec \neq \p M$).
There we also studied the method computationally up to the reconstruction of boundary distance functions. 

It is well-known \cite{Kurylev1997} that the 
boundary distance functions can be used to determine the
geometry (i.e. to determine the metric $g$ up to boundary fixing
isometries). 
While several methods to recover the geometry from the boundary distance functions have been proposed 
\cite{deHoop2014,Katchalov2001,Katsuda2007,Pestov2015},
these have not been implemented computationally to our knowledge.
It appears to us that, at least in the isotropic case, it is better to recover the wave speed directly without first recovering the boundary distance functions. 
In the next two sections, we will describe techniques that allow us to do so in both the isotropic and anisotropic cases. 
These will be based on the control problem setup from
\cite{Hoop2016} and we will recall the setup in this section.

The difference between \cite{Hoop2016}
and the present paper is that we do not use 
the sources $f$ solving the control problems of the form (\ref{abstract_cp})
to construct boundary distance functions, instead 
we will use them to recover information in the interior of $M$.
In the anisotropic case, this information is 
the internal data operator that gives wavefields
solving (\ref{wave_eq_aniso}) in semi-geodesic coordinates.

%In this section we give a brief overview of the techniques that we use
%from the Boundary Control method (BC method). We also highlight the
%notation and terminology which we use throughout this paper.

\subsection{Semi-geodesic coordinates and wave caps}

We consider an open subset $\Gamma \subset \p M$
satisfying 
$$
\{x \in \p M : d(x,\Gamma) \leq T\} \subset \Rec,
$$
where $d$ denotes the Riemannian distance associated with $g$.
We may replace $T$ by a smaller time to guarantee that there exists a non-empty $\Gamma$ satisfying this. 
In what follows we will only use the following further restriction of the N-to-D or D-to-N map
\begin{equation*}
  \Lambda_{\Gamma,\Rec}^{2T} f = \Lambda_{\Rec}^{2T} f|_{(0,2T) \times
    \Rec}, \quad f \in C_0^\infty((0,2T) \times \Gamma).
\end{equation*}

We now recall the definition of semi-geodesic coordinates associated to $\Gamma$. For
$y \in \Gamma$, we define $\sigma_{\Gamma}(y)$ to be the maximal arc length for which the normal geodesic beginning at $y$ minimizes the
distance to $\Gamma$. That is, letting $\gamma(s;y,v)$ denote the point at arc length
$s$ along the geodesic beginning at $y$ with initial velocity $v$, and
$\nu$ the inward pointing unit normal field on $\Gamma$, we define
\begin{equation*}
  \sigma_{\Gamma}(y) := \max \{ s \in (0, \tau_M(y, \nu)]:\ d(\gamma(s; y,
  \nu), \Gamma) = s\}.
\end{equation*}
We recall that $\sigma_{\Gamma}(y) > 0$ for $y \in \overline{\Gamma}$
(see e.g. \cite[p. 50]{Katchalov2001}). Defining,
\begin{equation}
  \textnormal{$x(y,s) := \gamma(s;y,\nu)$ \quad for $y \in \Gamma$ and $0 \leq
    s < \sigma_{\Gamma}(y)$,}
\end{equation}
the mapping
\begin{equation*}
  \Phi_g : \{(y,s) : y \in \Gamma \textnormal{ and } s \in [0, \sigma_\Gamma(y))\}  \to M,
\end{equation*}
given by $\Phi_g(y,s) := x(y,s)$ is a diffeomorphism onto its image in
$(M,g)$, and we refer to the pair $(y,s)$ as the semi-geodesic
coordinates of the point $x(y,s)$. We note that the semi-geodesic
``coordinates'' that we have defined here are not strictly coordinates
in the usual sense of the term, since they associate points in $M$
with points in $\R \times \Gamma$ instead of points in $\R^n$. To
obtain coordinates in the usual sense, one must specify local
coordinate charts on $\Gamma$.  Denoting the local coordinates on
$\Gamma$ associated with these charts by $(y^1,\ldots,y^{n-1})$, one
can then define local semi-geodesic coordinates by
$(y^1,\ldots,y^{n-1},s)$. We will continue to make this distinction,
using the term ``local'' only when we need coordinates in the usual
sense.

%% In order to sample wavefields in semi-geodesic coordinates, we will
%% make use of a family of sources producing waves that are approximately
%% constant in sets known as \emph{wave caps}. We first recall the
%% definition of a wave cap.
In both the scalar and anisotropic cases, our approach to recover
interior information relies on computing localized averages of
functions inside of $M$. One of the main components used to compute
these averages is a family of sources that solve blind control
problems with target functions of the form $\phi = 1_B$, where $B$ is
a set known as a \emph{wave cap}. The construction of these sources
will be recalled below in Lemma \ref{lemma:approxConstControl}, but
first we recall how wave caps are defined:
\begin{definition}
  Let $y \in \Gamma$, $s,h > 0$ with $s + h <
  \sigma_{\Gamma}(y)$. The \emph{wave cap,} $\wavecap_\Gamma(y,s,h)$, is defined as:
  \begin{equation*}
    \wavecap_\Gamma(y,s,h) := \{x \in M : d(x,y) \leq s + h \textnormal{ and } d(x,\Gamma) \geq s\}
  \end{equation*}
  See Figure \ref{fig:waveCapGeom} for an illustration.
\end{definition}

We recall that, for all $h > 0$, the point $x(y,s)$ belongs to the
set $\wavecap_\Gamma(y,s,h)$ and $\diam(\wavecap_\Gamma(y,s,h))
\rightarrow 0$ as $h \rightarrow 0$, (see e.g. \cite{Hoop2016}).  So,
when $h$ is small and $\phi$ is smooth, averaging $\phi$ over
$\wavecap_\Gamma(y,s,h)$ yields an approximation to
$\phi(x(y,s))$. These observations play a central role in our
reconstruction procedures.

\begin{figure}[!htb]
  \centering %\include{graphics}
  \includegraphics[width=3.50in]{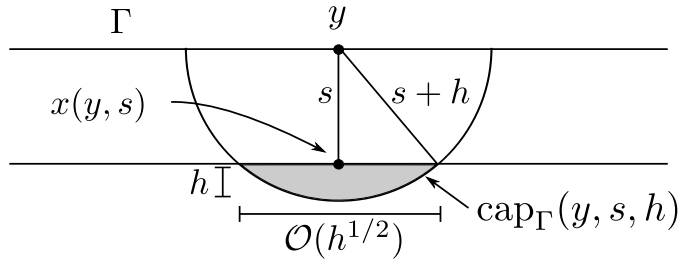}
  \caption{
    \label{fig:waveCapGeom}
    Geometry of a wave cap in the Euclidean case. In this case,
    Pythagoras' theorem suffices to show that
    $\diam(\wavecap_{\Gamma}(y,s,h)) = \mathcal{O}(h^{1/2})$, but this
    is also true in general.}
\end{figure}

\subsection{Elements of the BC method}

As mentioned above, the BC method involves finding sources $f$ for
which $u^f(T,\cdot) \approx \phi$ for appropriate functions $\phi \in
L^2(M)$. To that end, we recall the \emph{control map},
\begin{equation*}
  W : f \mapsto u^f(T,\cdot), \quad \textnormal{for $f \in L^2([0,T]\times
    \Gamma)$},
\end{equation*}
and note that $W$ is a bounded linear operator $W : L^2([0,T] \times
\Gamma) \rightarrow L^2(M)$, see e.g. \cite{Katchalov2001}. We remark
that the output of $W$ is a wave in the interior of $M$ and hence
cannot be observed directly from boundary measurements alone. Using
$W$, one defines the \emph{connecting operator} $K := W^*W$. The
adjoint here is defined with respect to the Riemannian volume measure
in the anisotropic case, and with respect to the scaled Lebesgue
measure $c^{-2}(x) dx$ in the isotropic case.  We denote these
measures by $\Vol$ in both cases.  In particular, we recall that $K$
can be computed by processing the N-to-D or D-to-N map via the
Blagovescenskii identity, see e.g. \cite{Liu2016}. That is,
\begin{equation}
\label{Blago}
  K = J \Lambda_{\Gamma}^{2T} \Theta -  R \Lambda_{\Gamma}^{T} R J \Theta,
\end{equation}
where $\Lambda_{\Gamma}^T f := (\Lambda_{\Gamma,\Rec}^T
f)|_{[0,T]\times\Gamma}$, $R f(t) := f(T -t)$ for $0 \leq t \leq T$,
$J f(t) := \int_t^{2T-t} f(s)\,ds$, and $\Theta $ is the inclusion
operator $\Theta : L^2([0,T] \times \Gamma) \hookrightarrow L^2([0,2T]
\times \Gamma)$ given by $\Theta f(t) = f(t)$ for $0 \leq t \leq T$
and $\Theta f(t) = 0$ otherwise. We remark that the Blagovescenskii
identity shows that $K$ can be computed by operations that only
involve manipulating the boundary data.

We recall some mapping properties of $W$ that follow from finite speed
of propagation for the wave equations (\ref{wave_eq_iso}) and
(\ref{wave_eq_aniso}). Let $\tau : \overline{\Gamma} \rightarrow
[0,T]$, and define $S_\tau := \{(t,y) : T - \tau(y) \leq t \leq
T\}$. Then, finite speed of propagation implies that if $f$ is a
boundary source supported in $S_\tau$, the wavefield $u^f(T,\cdot)$
will be supported in the domain of influence $M(\tau)$, defined by
\begin{equation*}
  M(\tau) := \{ x \in M : d(x,\Gamma) < \tau(y) \textnormal{ for some
    $y \in \Gamma$}\}.
\end{equation*}
In turn, this implies that $W$ satisfies, $W : L^2(S_\tau) \rightarrow
L^2(M(\tau))$.  So, if we define $P_\tau : L^2([0,T]\times\Gamma)
\rightarrow L^2(S_\tau)$, then we can define a restricted control map
$W_\tau := W P_\tau$, which satisfies $W_\tau : L^2(S_\tau)
\rightarrow L^2(M(\tau))$. The point here is that, although we do not
have access to the the output of $W_\tau$, we know that the waves will
be supported in the domain of influence $M(\tau)$. We also define the
restricted connecting operator $K_\tau := (W_\tau)^*W_\tau = P_\tau K
P_\tau$, and note that $K_\tau$ can be computed by first computing $K$
via (\ref{Blago}) and then applying the operator $P_\tau$.

To construct sources that produce approximately constant wavefields on
wave caps, we use a procedure from \cite{Hoop2016}. This procedure
uses the fact that a wave cap can be written as the difference of two
domains of influence, and requires that distances between boundary
points are known. Specifically, we will suppose that for any pair $x,y
\in \Gamma$ the distance $d(x,y)$ is known. As noted in
\cite{Hoop2016}, this is not a major restriction, since these
distances can be constructed from the data
$\Lambda_{\Gamma}^{2T}$. Then, using this collection of distances, we
define a family of functions $\tau_y^R : \overline{\Gamma} \rightarrow
\R_+$ by:
\begin{equation*}
  \textnormal{for $y \in \overline{\Gamma}$ and $R > 0$, define
    $\tau_y^R(x) := (R - d(x,y)) \vee 0$.}
\end{equation*}
Here we use the notation $\phi \vee \psi$ to denote the point-wise
maximum between $\phi$ and $\psi$, and we will continue to use this
notation below. Finally, one can show that $\wavecap_\Gamma(y,s,h) =
\overline{M(\tau_y^{s+h} \vee s1_\Gamma) \setminus M(s1_\Gamma)}$. We
also note that, since $\p M(\tau)$ has measure zero provided that $\tau$ is
continuous on $\overline{\Gamma}$ \cite{Oksanen2011}, one has that
$1_{\wavecap_\Gamma(y,s,h)} = 1_{M(\tau_y^{s+h} \vee s1_\Gamma)} -
1_{M(s1_\Gamma)}$ a.e.

The following lemma is an amalgamation of results from
\cite{Hoop2016}, and shows that there is a family of sources
$\psi_{h,\alpha}$ which produce approximately constant wavefields
$u^{\psi_{h,\alpha}}(T,\cdot)$ on wave caps, and that these sources
can be constructed from the boundary data
$\Lambda_{\Gamma,\Rec}^{2T}$.

\begin{lemma}
  \label{lemma:approxConstControl}
  Let $y \in \Gamma$, $s,h > 0$ with $s + h < \sigma_{\Gamma}(y)$. Let
  $\tau_1 = s 1_\Gamma$ and $\tau_2 = \tau_{y}^{s+h} \vee
  s1_\Gamma$. Define $b(t,y) := T - t$, and let $\widetilde{b} = b$ in
  the Neumann case, and $\widetilde{b} = (\Lambda_{\Gamma,\Rec}^T)^*b$
  in the Dirichlet case. Then, for each $\alpha > 0$, let $f_{\alpha,i} \in
  L^2(S_{\tau_i})$ be the unique solution to
  %%%%%%%%%%%%%%%%%%%%%%%%%%%%%%%%%%%%%%%%%%%%%%%%%%%%%%%%%%%%%%%%%%%%%%%%%%%%%%
  %% Use for D-to-N case:
  %%%%%%%%%%%%%%%%%%%%%%%%%%%%%%%%%%%%%%%%%%%%%%%%%%%%%%%%%%%%%%%%%%%%%%%%%%%%%%
  %% \begin{equation}    
  %%   \label{eqn:ControlProblem}
  %%   \left(K_{\tau_i} + \alpha\right) f = P_{\tau_i} \widetilde{b},
  %% \end{equation}
  %% where $\widetilde{b} := (\Lambda_{\Gamma,\Rec}^T)^*b$ and $b(t,y) := T
  %% - t$. Define,
  %%%%%%%%%%%%%%%%%%%%%%%%%%%%%%%%%%%%%%%%%%%%%%%%%%%%%%%%%%%%%%%%%%%%%%%%%%%%%%
  %% Use for N-to-D case:
  %%%%%%%%%%%%%%%%%%%%%%%%%%%%%%%%%%%%%%%%%%%%%%%%%%%%%%%%%%%%%%%%%%%%%%%%%%%%%%
  \begin{equation}    
    \label{eqn:ControlProblem}
    \left(K_{\tau_i} + \alpha\right) f = P_{\tau_i} \widetilde{b}.
  \end{equation}
  Define,
  \begin{equation}
    \label{eqn:definePsi}
    \psi_{h,\alpha} = f_{\alpha,2} - f_{\alpha,1}.
  \end{equation}
  Using the notation $B_h = \wavecap_\Gamma(y,s,h)$, it holds that
  \begin{equation}
    \label{eqn:limitingExpression}
    \lim_{\alpha \rightarrow 0} u^{\psi_{h,\alpha}}(T,\cdot) = 1_{B_h}
    \textnormal{\quad and \quad}
    \lim_{\alpha \rightarrow 0} \langle \psi_{h,\alpha}, P_{\tau_2} b\rangle_{L^2(S_\tau)} = \Vol(B_h).
  \end{equation}
\end{lemma}

We briefly sketch the proof of Lemma
\ref{lemma:approxConstControl}. The main idea is to approximately
solve the blind control problem (\ref{abstract_cp}) with $\phi \equiv
1$ over the spaces $L^2(S_{\tau_i})$ for $i=1,2$. To accomplish this,
for $i=1,2$, one can consider a Tikhonov regularized version of
(\ref{abstract_cp}) depending upon a small parameter $\alpha >
0$. Then, letting $f_{\alpha,i}$ denote the minimum of the associated
Tikhonov functional for $\alpha > 0$, one can obtain $f_{\alpha,i}$ by
solving this functional's normal equation, given by
(\ref{eqn:ControlProblem}). Note that all of the terms defining
$f_{\alpha,i}$ in (\ref{eqn:ControlProblem}) can be computed in terms
of the boundary data, so $f_{\alpha,i}$ can be obtained without
knowing the wavespeed or metric. Appealing to properties of Tikhonov
minimizers, one can then show that $Wf_{\alpha,i} \rightarrow
1_{M(\tau_i)}$ as $\alpha \rightarrow 0$, and hence $W\psi_{\alpha,h}
= Wf_{\alpha,1} - Wf_{\alpha,2} \rightarrow 1_{M(\tau_2)} -
1_{M(\tau_1)} = 1_{\wavecap_\Gamma(y,s,h)}$, where each limit and
equality holds in the $L^2$ sense.

%% \section{Alternative way to organize the computations}
%% To avoid the artefacts observed in \cite{Hoop2016},
%% we propose to use optimize as follows:
%% \begin{enumerate}
%% \item Find $h$ such that $u^h(T) \approx 1_{M(\tau_y^{s+h})}$
%% \item Find $\tilde f$ such that $u^{\tilde f}(T) \approx 1_{M(s 1_\Gamma)} u^{f}(T)$. (In the isotropic case, s.t. $u^{\tilde f}(T) \approx 1_{M(s 1_\Gamma)} \phi$ for harmonic $\phi$.)
%% \end{enumerate}
%% Then 
%% $$
%% \pair{h, K (f-\tilde f)} = \pair{1_{B_h}, u^f(T)}.
%% $$
%% Note that when computing the distances in \cite{Hoop2016} we needed intersections of three domains of influence, but here we need only intersections of two domains of influence. This allows us to organize the computations so that the worst artefacts, due to instability of the control problem (\ref{abstract_cp}) as observed in \cite{Hoop2016}, are avoided.

\section{Recovery of information in the interior}
\label{sec:interior}

Propositions \ref{prop:wvfldReconstr}
and \ref{prop:harmonicReconstr} below
can be viewed as variants of Corollaries 1 and 2 in 
\cite{Bingham2008}, the difference being that we 
use the control problem setup discussed in the previous section.
One advantage of this setup is that we do not 
need to make the auxiliary assumption that the limit (14) in \cite{Bingham2008} is non-zero.

\subsection{Wave field reconstruction in the anisotropic case}

We begin with reconstruction of wavefields sampled in semi-geodesic
coordinates, as encoded by the following map.

\begin{definition} 
  Let $(y,s)\in\domain(\Phi_g)$ and $f \in L^2([0,T]\times\Gamma)$.
  The map $L_g : L^2([0,T]\times\Gamma) \rightarrow
  L^2(\domain(\Phi_g))$ is defined pointwise by
  \begin{equation}
    L_g f(y,s) := u^f(T,x(y,s)).
  \end{equation}
\end{definition}

We now show that $L_g$ can be computed from the N-to-D map.

\begin{proposition}
  \label{prop:wvfldReconstr}
  Let $f \in C_0^\infty([0,T] \times \Gamma)$. Let $t \in [0,T]$, $y
  \in \Gamma$ and $s,h > 0$ with $s+h < \sigma_\Gamma(y)$ and $h$
  sufficiently small.  The family of sources
  $\{\psi_{h,\alpha}\}_{\alpha > 0}$ given in Lemma
  \ref{lemma:approxConstControl} satisfies
  \begin{equation}
    \label{eqn:accuracyOfApprox}
    \lim_{\alpha \rightarrow 0} \frac{\langle \psi_{h,\alpha}, K f \rangle_{L^2([0,T]\times\Gamma)} }{ \langle \psi_{h,\alpha}, P_\tau b \rangle_{L^2([0,T]\times\Gamma)} } =
     u^{f}(t,x(y,s)) + \mathcal{O}(h^{1/2}).
  \end{equation}
\end{proposition}
\begin{proof}
  Applying Lemma \ref{lemma:approxConstControl}, we have that
  \begin{equation*}
    \lim_{\alpha \rightarrow 0} \frac{\langle \psi_{h,\alpha}, K f \rangle_{L^2(S_\tau)}}{\langle \psi_{h,\alpha}, P_\tau b \rangle_{L^2(S_{\tau})}} =
    \frac{\lim_{\alpha \rightarrow 0} \langle W \psi_{h,\alpha}, W f \rangle_{L^2(M)}}{ \lim_{\alpha \rightarrow 0} \langle \psi_{h,\alpha}, P_\tau b \rangle_{L^2(S_{\tau})}}
    = \frac{\langle 1_{B_h}, u^f(T,\cdot) \rangle_{L^2(M)}}{\Vol(B_h)}.
  \end{equation*}
  Thus it suffices to show that: 
  \begin{equation*}
    \langle 1_{B_h}, u^f(T,\cdot) \rangle = \Vol(B_h) u^{f}(T,x(y,s)) + \Vol(B_h) \mathcal{O}(h^{1/2}).
  \end{equation*}

  Suppose that $h$ is sufficiently small that $B_h$ is contained in
  the image of a coordinate chart $(p, U)$ (that is, we use the
  convention that $p : U \subset \R^n \rightarrow p(U) \subset
  M$). We denote the coordinates on this chart by
  $(x^1,\ldots,x^n)$, and also suppose that $x(y,s)$ corresponds to
  the origin in this coordinate chart. Since $f$ is $C_0^\infty$, it
  follows that $u^f$ is smooth. Thus we can Taylor expand
  $u^f(T,\cdot)$ in coordinates about $x(y,s) \in B_h$, giving,
  \begin{equation*}
    \qquad u^f(T, x^1,\ldots,x^n) = u^f(T, 0,\ldots,0) + \partial_i u^f(T, 0,\ldots,0) x^i + \sum_{|\beta| = 2} R_\beta(x^1,\ldots,x^n) x^{\beta}
  \end{equation*}
  Where $R_\beta$ is bounded by the $C^2$ norm of
  $u^f(T,x^1,\ldots,x^n)$ (i.e. of $u^f(T,\cdot)$ in coordinates), on
  any compact neighborhood $K$ satisfying $0 \in K\subset U$. In
  particular we choose $K$ such that $B_h \subset p(K)$ for $h$
  sufficiently small. Combining these expressions and using that
  $x(y,s)$ corresponds to $0$ in $U$,
  \begin{align*}
    \left|\langle 1_{B_h},  u^f(T,\cdot) \rangle_{L^2(M)} - \Vol(B_h) u^f(T, x(y,s))\right|&  \\
    \leq C \int_{p^{-1}(B_h)} |\partial_i u^f(T, 0,\ldots,0) x^i|  &+ \sum_{|\beta| = 2} |R_\beta(x^1,\ldots,x^n) x^{\beta_1} x^{\beta_2}| \,dx^1 \cdots dx^n
  \end{align*}
  Then for points $p(x) \in M$ with coordinates $x \in U$ sufficiently
  close to $0$, there exist constants $g_*, g^*$ such that $g_* |x|_e
  \leq d(p(x),0) \leq g^* |x|_e$, where $|x|_e$ denotes the Euclidean
  length of the coordinate vector $x$ in $\R^n$. So, let $x =
  (0,\ldots,x^i,\ldots,0)$, then note that $|x^i| = |x|_e \leq (1/g_*) d(0, p(x))
  \leq (1/g_*) \diam(B_h)$. Thus, for $h$ sufficiently small,
  \begin{align*}
    |\langle 1_{B_h},  u^f(T,\cdot) \rangle_{L^2(M)} - \Vol(B_h) u^f(T, x(y,s))|&  \\
    \leq \|u^f\|_{C^1(K)} C\diam(B_h) \Vol(B_h)  &+  \|u^f\|_{C^2(K)} (C\diam(B_h))^2 \Vol(B_h)
  \end{align*}  
  Finally, the discussion in \cite{Bingham2008} implies that
  $\diam(B_h) = \mathcal{O}(h^{1/2})$, which completes the proof.\qquad
\end{proof}

\begin{corollary}
  \label{corr:wvfldReconstr}
  For each $f \in C_0^\infty([0,T]\times\Gamma)$, $L_g f$ can be
  determined pointwise by taking the limit as $h \rightarrow 0$ in
  (\ref{eqn:accuracyOfApprox}). Since $C_0^\infty([0,T]\times\Gamma)$
  is dense in $L^2([0,T]\times\Gamma)$ and $L_g$ is bounded on
  $L^2([0,T]\times\Gamma)$, we have that $L_g f$ is determined for all $f
  \in L^2([0,T]\times\Gamma)$.
  %%%%%%%%%%%%%%%%%%%%%%%%%%%%%%%%%%%%%%%%%%%%%%%%%%%%%%%%%%%%%%
  %% OLD:
  %%%%%%%%%%%%%%%%%%%%%%%%%%%%%%%%%%%%%%%%%%%%%%%%%%%%%%%%%%%%%%
  %%The map $L_g$ can be constructed from  by taking limits of the data obtained
  %%from (\ref{lemma:wvfldReconstr}).
\end{corollary}
\begin{proof}
  First, let $f \in C_0^\infty([0,T]\times\Gamma)$. Taking the limit
  as $h \rightarrow 0$ in the preceding lemma shows that $L_g f(y,s)$ can
  be computed for any pair $(y,s) \in \domain(\Phi_g)$, and thus $L_g f$
  can be determined in semi-geodesic coordinates.

  Now we show that $L_g f$ can be determined for any $f \in
  L^2([0,T]\times\Gamma)$. First we recall that $L_g f = \Phi_g^*
  Wf$. Since the pull-back operator $\Phi_g^*$ just composes a
  function with a diffeomorphism, and $\overline{\Gamma}$ is compact,
  we have that $\Phi_g^*$ is bounded as an operator $\Phi_g^* :
  L^2(\range(\Phi_g)) \rightarrow L^2(\domain(\Phi_g))$. Thus $L_g$ is a
  composition of bounded operators, and hence $L_g : L^2([0,T]\times
  \Gamma) \rightarrow L^2(\domain(\Phi_g))$ is bounded. Let $f \in
  L^2([0,T]\times\Gamma)$ be arbitrary. Since
  $C_0^\infty([0,T]\times\Gamma)$ is dense in $L^2$ one can find a
  sequence $\{f_j\}_{j=1}^\infty \subset
  C_0^\infty([0,T]\times\Gamma)$ such that $f_j \rightarrow f$. Then,
  since $L_g$ is bounded, $L_g f = \lim_{j \rightarrow \infty} L_g f_j$. \qquad
\end{proof}

\subsection{Coordinate transformation reconstruction in the isotropic case}

The map $\Lambda_{\p M}^T$ is invariant under diffeomorphisms that fix the boundary of $M$, and therefore in the anisotropic case it is not possible to compute $g$ in the Cartesian coordinates. The same is true for the wavefields. 
In the isotropic case, on the other hand, it is possible to compute the map $\Phi_g(y,s)$, and in fact, 
the wave speed was determined in 
Belishev's original paper \cite{Belishev1987} by 
first showing that the internal data $u^f(t,x)$ can be recovered in the Cartesian coordinates, and then using the identity
$$
\frac{\Delta u(t,x)}{\p_t^2 u(t,x)} = c^{-2}(x).
$$
It was later observed that the wave speed can be recovered directly from the map $\Phi_g$ without using information on the wavefields in the interior, see e.g. \cite{Belishev1999,Bingham2008}. 
In the present paper we will compute $\Phi_g(y,s)$ by applying the following lemma to the Cartesian coordinate functions.

\begin{proposition}
  \label{prop:harmonicReconstr}
Suppose that $g$ is isotropic, that is, $g = c^{-2}(x) dx^2$.  Let
$\phi \in C^\infty(M)$ be harmonic, that is, $\Delta \phi = 0$.  Let
$t \in [0,T]$, $y \in \Gamma$, and $s,h > 0$ with $s+h <
\sigma_\Gamma(y)$.  Then, for $h$ small, the family of sources
$\{\psi_{h,\alpha}\}_{\alpha > 0}$ given in Lemma
\ref{lemma:approxConstControl} satisfies
  \begin{equation}
    \label{eqn:accuracyOfApproxHarm}
    \lim_{\alpha \rightarrow 0} \frac{ B(\psi_{h,\alpha},  \phi) }{ B(\psi_{h,\alpha}, 1) } =
     \phi(x(y,s)) + \mathcal{O}(h^{1/2}),
  \end{equation}
where 
\begin{equation}
  \label{define_B}
  B(f,\phi)  = \langle f, b
   \phi \rangle_{L^2([0,T]\times\Gamma; dy)} -
  \langle\Lambda^T_{\Gamma,\Rec} f, b\p_\nu \phi
  \rangle_{L^2([0,T]\times\Rec; dy)}.
\end{equation}
Where $b(t) = T - t$.
\end{proposition}
\begin{proof}
The proof is analogous to that of Proposition \ref{prop:wvfldReconstr} after observing that 
  \begin{equation*}
    \lim_{\alpha \rightarrow 0} \frac{B(\psi_{h,\alpha},  \phi)}{ B(\psi_{h,\alpha}, 1)} 
= \frac{\langle 1_{B_h}, \phi \rangle_{L^2(M;c^{-2} dx)}}{\langle 1_{B_h}, 1 \rangle_{L^2(M;c^{-2} dx)}}.
\end{equation*}
To see this, it suffices to show that for $\phi$ harmonic and $f \in
L^2([0,T]\times\Gamma)$,
\begin{equation}
  \label{eqn:Bintermed}
  B(f,\phi) = \langle u^f(T),\phi\rangle_{L^2(M; c^{-2} dx)},
\end{equation}
since then
\begin{equation*}
  \lim_{\alpha \rightarrow 0} B(\psi_{h,\alpha},\phi) = \lim_{\alpha \rightarrow 0} \langle u^{\psi_{h,\alpha}}(T),\phi\rangle_{L^2(M; c^{-2} dx)} = \langle 1_{B_h}, \phi \rangle_{L^2(M;c^{-2} dx)}.
\end{equation*}
This expression holds, in particular, for the special case that $\phi
\equiv 1$, since constant functions are harmonic.

The demonstration of (\ref{eqn:Bintermed}) is known, and is based upon
the following computation,
\begin{equation*}
  \begin{array}{rl}
    \p_t^2 \langle u^f(t), \phi\rangle_{L^2(M; c^{-2} dx)} &= \langle \Delta u^f(t), \phi\rangle_{L^2(M; dx)} - \langle u^f(t), \Delta \phi\rangle_{L^2(M; dx)}\\
    &= \langle f(t), \phi\rangle_{L^2(\p M; dy)} - \langle\Lambda f(t), \p_\nu \phi\rangle_{L^2(\p M; dy)},
  \end{array}
\end{equation*}
where we have written $\Lambda f = u^f|_{\p M}$. Thus, the map $t
\mapsto \langle u^f(t), \phi \rangle$ satisfies an ordinary
differential equation with vanishing initial conditions, since $u^f(0)
= \p_t u^f(0) = 0$. Solving this differential equation and evaluating
the result at $t = T$, we get an explicit formula for $\langle u^f(T),
\phi\rangle$ depending upon $f$ and $\Lambda f$:
\begin{equation}
  \label{eqn:harmonicBlagoType}
  \langle u^f(T), \phi \rangle_{L^2(M; c^{-2} dx)} = \langle f, b
   \phi \rangle_{L^2([0,T]\times\Gamma; dy)} -
  \langle\Lambda^T_{\Gamma,\Rec} f, b\p_\nu \phi
  \rangle_{L^2([0,T]\times\Rec; dy)}.
\end{equation}
%% OLD:
%% \begin{equation*}
%%   \begin{array}{rl}
%%     \langle u^f(T), \phi \rangle_{L^2(M; c^{-2} dx)} &= \int_0^T  \langle f(t), b(t)\phi \rangle_{L^2(\p M; dy)} - \langle\Lambda f(t), b(t)\p_\nu \phi\rangle_{L^2(\p M; dy)}  \,dt\\
%%     &= \langle f, b \otimes \phi \rangle_{L^2([0,T]\times\Gamma; dy)} - \langle\Lambda^T_{\Gamma,\Rec} f, b\otimes\p_\nu \phi \rangle_{L^2([0,T]\times\Rec; dy)}.
%%   \end{array}
%% \end{equation*}
Which completes the demonstration of (\ref{eqn:Bintermed}). Notice
that we only require $\Lambda f|_\Rec$, since, for $t \in [0,T]$,
$\Lambda f(t)$ vanishes outside of $\Rec$ by finite speed of
propagation. An analogous derivation can be found in \cite{Liu2016}
(with full boundary measurements and the D-to-N map instead of the
N-to-D map).
\end{proof}

As in Corollary \ref{corr:wvfldReconstr}, letting $h \to 0$, we see that the map 
$$
H_c : \{\phi \in C^\infty(M);\ \Delta \phi = 0\} \to C^\infty(\domain(\Phi_g)), \quad
H_c \phi(y,s) = \phi(\Phi_g(y,s)),
$$
can be computed from the N-to-D map, where $g = c^{-2}(x) dx^2$.
To see this, first recall that $\Phi_g(y,s) := \gamma(s; y, \nu)$. Since $\gamma(\cdot;y,\nu)$ is a geodesic and $\nu$ has unit-length vector with respect to the metric $g$, we have that 
$|\p_s \Phi_g(y,s)|_g = 1$.
Then, recall that for $x \in M$ and $v \in T_x M$, the length $|v|_g$ is computed by 
$
|v|_g^2 = c(x)^{-2} |v|_e^2,
$
where $|v|_e$ is the Euclidean length of $v$.
Writing $x^j$, $j=1,\dots,n$, for the Cartesian coordinate functions on $M$, it follows that
\begin{equation}
\label{scheme_isotropic}
\hspace{-.5cm}\Phi_g(y,s) = (H_c x^1(y,s), \dots, H_c x^n(y,s)), 
\quad c(\Phi_g(y,s))^2 = |\p_s \Phi_g(y,s)|_e^2.
\end{equation}
Thus $c$ can be computed in the Cartesian coordinates 
by inverting the first function above and composing the inverse with the second function.
We will show in Section \ref{sec_stability} that this simple inversion method is stable.

The recovery of the internal information encoded by $L_g$ and $H_c$ is the most unstable part of the Boundary Control method as used in this paper. 
The convergence with respect to $h$ is sublinear
as characterized by (\ref{eqn:accuracyOfApprox}) and (\ref{eqn:accuracyOfApproxHarm}), 
and the convergence with respect to $\alpha$ is even worse. 
In general we expect it to be no better than logarithmic.
The recent results 
\cite{Bosi2016,Laurent2015}
prove logarithmic stability for related control and unique continuation problems, and 
\cite{Hoop2016} describes how the instability shows up in numerical examples. 

\section{Recovery of the metric tensor}

Due to the diffeomorphism invariance discussed above, we cannot recover $g$ in the Cartesian coordinates and it is natural to recover $g$ in the semi-geodesic coordinates.
This is straightforward in theory when the internal information $L_g$ is known, and analogously to the elliptic inverse problems with internal data \cite{Bal2013},
we expect that the problem has good stability properties when suitable sources $f$ are used. 
We will next describe a way to choose the sources by using an optimization technique. 
This technique is not stable in general, but 
as shown in the next section, stability holds under suitable convexity assumptions.

\begin{lemma}
  In any local coordinates $(x^1, \dots, x^n)$,
  \begin{equation}
    \label{eqn:getMetricFromLaplacians}
    g^{lk}(x) = \frac{1}{2}\left(\Delta_{g}(x^l x^k) - x^k \Delta_{g}x^l  - x^l \Delta_{g}x^k\right).
  \end{equation}
\end{lemma}
%%\begin{proof}
{\em Proof.}  Let $(x^1,\ldots,x^n)$ be local coordinates on
$M$. Write $\alpha := \sqrt{g}$. Then
\begin{align*}
   \Delta_{g}(x^l x^k) &= \frac{1}{\alpha} \p_i\left(\alpha g^{ij}  \p_j (x^l x^k)\right)  \\
   &= \frac{1}{\alpha} \left(\p_i\left(\alpha g^{il}  x^k\right)  + \p_i\left(\alpha g^{ik}  x^l\right) \right) \\
   &= g^{kl}  + x^k \frac{1}{\alpha} \p_i\left(\alpha g^{il}\right)  + g^{lk}  + x^l \frac{1}{\alpha} \p_i\left(\alpha g^{ik}\right)  \\
   &= 2 g^{lk}  + x^k \Delta_{g}x^l  + x^l \Delta_{g}x^k. \qquad \textnormal{\qed}
\end{align*}

\begin{proposition}
  \label{prop:computingG}
  The metric $g$ can be constructed in local semi-geodesic
  coordinates using the operator $L_g$ as data.
\end{proposition}
\begin{proof}
Let $\Omega = \range(\Phi_g)$, and $\omega \subset
  \Omega$ be a coordinate neighborhood for the semi-geodesic
  coordinates. Let $(x^1,\ldots,x^n)$ denote local semi-geodesic
  coordinates on $\omega$. Fix $1\leq j,k \leq n$ and for 
  $\ell = 1,2,3$ choose $\phi_\ell \in
C_0^\infty(\Omega)$, $\ell = 1,2,3$, such that for all $x \in
\omega$,
\begin{equation}
\label{def_phi_ell}
  \phi_1(x) = x^j x^k, \quad \phi_2(x) = x^j, \quad \phi_3(x) = x^k. 
\end{equation}
Consider the following Tikhonov
  regularized problem: for $\alpha > 0$ find $f \in L^2([0,T]\times
  \Gamma)$ minimizing
  \begin{equation*}
    \|L_g f - \phi^\ell\|_{L^2(\Omega)}^2 + \alpha \|f\|_{L^2([0,T]\times\Gamma)}^2.
  \end{equation*}
It  is a well known consequence of \cite{Tataru1995}, see e.g. 
\cite{Katchalov2001}, that $L_g$ has dense range in $L^2(\Omega)$.
Thus this problem has a minimizer
  $f_{\alpha,\ell}$ which can be obtained as the unique solution to
  the normal equation, see e.g. \cite[Th. 2.11]{Kirsch2011},
  \begin{equation}
  \label{scheme_anisotropic}
    (L_g^*L_g + \alpha) f = L_g^* \phi^\ell.
  \end{equation}
  It follows from \cite[Lemma 1]{Oksanen2013} that the
  minimizers satisfy
  \begin{equation*}
    \lim_{\alpha \rightarrow 0} L_g f_{\alpha,\ell} = \phi^{\ell}.
  \end{equation*}
As the wave equation (\ref{wave_eq_aniso})
is translation invariant in time, we have $L_g \p_t^2 f =
  \Delta_g u^f(T, \cdot)$, and therefore
  \begin{align*}    
    \lim_{\alpha\rightarrow0}\|L_g \p_t^2 f_{\alpha,\ell} - \Delta_g \phi^\ell\|_{H^{-2}(\Omega)} &= \lim_{\alpha\rightarrow0}\|\Delta_g(u^{f_{\alpha,\ell}}(T,\cdot) - \phi^\ell)\|_{H^{-2}(\Omega)}\\ &\le C \lim_{\alpha\rightarrow0}\|u^{f_{\alpha,\ell}}(T,\cdot) - \phi^\ell\|_{L^2(\Omega)} = 0.
  \end{align*}
  Thus for $\ell = 1,2,3$, $L_g \p_t^2 f_{\alpha,\ell} \rightarrow
  \Delta_g\phi^\ell$ in the $H^{-2}(\Omega)$ sense. Using expression
  (\ref{eqn:getMetricFromLaplacians}), and recalling the definitions
  of the target functions $\phi^{\ell}$, then in the local coordinates
  on $\omega$ we have
  \begin{equation}
 \label{scheme_anisotropic_step2}
    g^{jk} = \lim_{\alpha \rightarrow 0} \frac{1}{2}(L_g\p_t^2{f_{\alpha,1}} - x^k L_g\p_t^2 f_{\alpha,2} - x^j L_g\p_t^2 f_{\alpha,3}),
  \end{equation}
  where the convergence is in $H^{-2}(\omega)$. Finally, since
  $\Omega$ can be covered with coordinate neighborhoods such as
  $\omega$, this argument can be repeated to determine $g^{lk}$ in any
  local semi-geodesic coordinate chart.\qquad
\end{proof}

\section{On stability of the reconstruction from internal data}
\label{sec_stability}

When discussing
stability near the set $\Gamma$, 
we will restrict our attention to $\Omega \subset
M$ and a set $\mathcal G$ of smooth Riemannian metrics on $M$ for which
\begin{equation}
  \label{Omega_cond} 
  \overline{\Omega} \subset \Phi_{\tilde g}(\Gamma \times [0, r_0)), \quad \tilde g \in \mathcal G,
\end{equation}
where $r_0 > 0$ is fixed. 

We begin by showing the following consequence of the implicit function theorem.

\begin{lemma}
\label{lem_Phi_inv}
Let $U \subset \R^n$ be open 
and let $\Phi_0 : \overline U \to \R^n$ be continuously differentiable.
Let $p_0 \in U$ and suppose that the derivative $D\Phi_0$ is ivertible at $p_0$. Then there are neighbourhoods $W \subset \R^n$ of $\Phi_0(p_0)$ 
and $\mathcal U \subset C^1(\overline U)$ of $\Phi_0$
such that 
$$
\norm{\Phi^{-1} - \Phi_0^{-1}}_{C^0(\overline W)}
\le C \norm{\Phi - \Phi_0}_{C^1(\overline U)},
\quad \Phi \in \mathcal U.
$$
\end{lemma}
\begin{proof}
Define the map $$
F : C^1(\overline U) \times \R^n \times \R^n \to \R^n, \quad
F(\Phi, q, p) = \Phi(p) - q.
$$
Then $F$ is continuously differentiable, and 
$
D_p F(\Phi_0, p_0) = D \Phi_0(p_0).
$
Thus the implicit function theorem, see e.g. \cite[Th 6.2.1]{Lang1983},
implies that there are neighbourhoods $V, W' \subset \R^n$ of $p_0, \Phi_0(p_0)$ 
and $\mathcal U' \subset C^1(\overline U)$ of $\Phi_0$,
and a continuously differentiable map 
$H : \mathcal U' \times W' \to V$ such that 
$F(\Phi, q, H(\Phi, q)) = 0$.
But this means that $H(\Phi, \cdot) = \Phi^{-1}$ in $W'$.
Choose a neighbourhood $W$ of $\Phi_0(p_0)$ such that $\overline W \subset W'$ and that $\overline W$ is compact.
As $H$ is continuously differentiable, there is a neighbourhood $\mathcal U \subset \mathcal U'$ of $\Phi_0$ such that 
$$
|H(\Phi, q) - H(\Phi_0, q)| \le 2 \max_{q \in \overline W}\norm{D_\Phi H(\Phi_0, q)}_{C^1(\overline U) \to \R^n}
\norm{\Phi - \Phi_0}_{C^1(\overline U)}, \quad \Phi \in \mathcal U.
$$
\end{proof}

We have the following stability result in the isotropic case. 
\begin{theorem}
  \label{th_main_iso}
Consider a family $\mathcal G$ of smooth isotropic metrics $\tilde g = \tilde c^{-2} dx^2$
satisfying (\ref{Omega_cond}). Let $c^{-2} dx^2 \in \mathcal G$ and suppose that
  \begin{equation}
    \label{smallness_C1_iso}
    \norm{\tilde c - c}_{C^2(M)} \le \epsilon, \quad \tilde c^{-2} dx^2 \in \mathcal G.
  \end{equation}
 Then for small enough $\epsilon > 0$, there is $C > 0$ such that 
  \begin{equation*}
    \norm{\tilde c^2 - c^2}_{C(\Omega)} \le C \norm{H_{\tilde c} - H_c}_{C^1(M) \to C^1(\Gamma \times [0,r_0))}.
  \end{equation*}
\end{theorem}
\begin{proof}
We write $\Sigma = \Gamma \times (0,r_0)$, $\tilde g = \tilde c^{-2} dx^2$ and $g = c^{-2} dx^2$.
Then (\ref{scheme_isotropic}) implies that 
\begin{eqnarray*}
\norm{\Phi_{\tilde g} - \Phi_g}_{C^1(\Sigma)}
&\le& C \norm{H_{\tilde c} - H_c}_{C^1(M) \to C^1(\Sigma)}.
\end{eqnarray*}
Moreover, again by (\ref{scheme_isotropic}),
\begin{eqnarray*}
\norm{\tilde c^2 \circ \Phi_{\tilde g} - c^2 \circ \Phi_g}_{C^0(\Sigma)} &\le& 
C \norm{H_{\tilde c} - H_c}_{C^1(M) \to C^1(\Sigma)}.
\end{eqnarray*}
This together with  
\begin{eqnarray*}
\norm{\tilde c^2 - c^2}_{C^0(\Omega)} &\le& 
\norm{\tilde c^2 \circ \Phi_{\tilde g}\circ \Phi_{\tilde g}^{-1} - c^2 \circ \Phi_g\circ \Phi_{\tilde g}^{-1}}_{C^0(\Omega)} 
\\&&\quad+ 
\norm{c^2 \circ \Phi_g\circ \Phi_{\tilde g}^{-1} - c^2 \circ \Phi_g\circ \Phi_g^{-1}}_{C^0(\Omega)}
\end{eqnarray*}
implies that it is enough to study 
%$\norm{\Phi_{\tilde g}^{-1}}_{C^0(\Omega)}$ and 
$\norm{\Phi_{\tilde g}^{-1}- \Phi_g^{-1}}_{C^0(\Omega)}$.

Note that $(\tilde g, y, s) \mapsto \Phi_{\tilde g}(y,s)$ is continuously differentiable
since it is obtained by solving the ordinary differential equation that gives the geodesics with respect to $\tilde g$.
Indeed, this follows from \cite[Th. 6.5.2]{Lang1983} by considering the vector field 
$F$
that generates the geodesic flow. In any local coordinates,
$F(x,\xi, h) = (\xi, f(x, \xi, h), 0)$
where $f = (f^1, \dots, f^n)$, 
$
f^j(x, \xi, h) = -\Gamma_{k\ell}^j(x,h)\xi^k \xi^\ell, 
$
and $\Gamma_{k\ell}^j(x,h)$ are the Christoffel symbols of a
metric tensor $h$ at $x$, that is,
$$
\Gamma_{k\ell}^j(x,h) = 
\frac{1}{2}h^{jm} \left(\frac{\partial h_{mk}}{\partial x^\ell} + \frac{\partial h_{m\ell}}{\partial x^k} - \frac{\partial h_{k\ell}}{\partial x^m} \right).
$$
In particular, if $\omega$ is a neighbourhood of $p_0 \in \Sigma$ and $\overline\omega \subset \Sigma$,
then 
the map $\tilde c \mapsto \Phi_{\tilde g}$
is continuous from $C^2(M)$ to $C^1(\overline \omega)$.
Thus, for small enough $\epsilon > 0$ in (\ref{smallness_C1_iso}), we may apply   
Lemma \ref{lem_Phi_inv} to obtain 
$$
\norm{\Phi_{\tilde g}^{-1}- \Phi_g^{-1}}_{C^0(W)}
\le C \norm{\Phi_{\tilde g} - \Phi_g}_{C^1(\Sigma)},
$$
where $W$ is a neighbourhood of $\Phi_g(p_0)$.
As $\overline \Omega$ is compact, it can be covered by a finite number of sets like the above set $W$. Thus 
$$
\norm{\Phi_{\tilde g}^{-1}- \Phi_g^{-1}}_{C^0(\Omega)}
\le C \norm{\Phi_{\tilde g} - \Phi_g}_{C^1(\Sigma)}
\le C \norm{H_{\tilde c} - H_c}_{C^1(M) \to C^1(\Sigma)}.
$$
%The continuity of the map $\tilde c \mapsto \Phi_{\tilde g}$ implies also that $\norm{\Phi_{\tilde g}^{-1}}_{C^0(\Omega)}$
%is uniformly bounded for $\tilde c$ satisfying (\ref{smallness_C1_iso}).
\end{proof}

We now consider the anisotropic case, and describe a geometric condition on $(M,g)$ that will yield
stable recovery of $g$ in the semi-geodesic coordinates of $\Gamma$
from $L_g$ in the set $\Omega$. Specifically, we
will assume that the following problem, which is the dual problem to (\ref{wave_eq_aniso}),
\begin{equation}
  \label{wave_eq_ad}
  \begin{array}{rcl}
    \p_t^2 w - \Delta_g w &=& 0, \quad  \textnormal{in $(0,T) \times M$},\\
    w|_{x \in \p M} &=& 0\\
    w|_{t=T} = 0,\  \p_t w|_{t=T} &=& \phi.
  \end{array}
\end{equation}
is \emph{stably observable} in the following sense.
\begin{definition} 
  \label{def:stable_obs}
  Let $\mathcal G$ be a subset of smooth Riemannian metrics on
  $M$. Then, (\ref{wave_eq_ad}) is \emph{stably observable} for
  $\Omega$ and $\mathcal G$ from $\Gamma$ in time $T > 0$ if there is
  a constant $C > 0$ such that for all $g \in \mathcal G$ and for all
  $\phi \in L^2(\Omega)$ the solutions $w = w^{\phi} = w^{\phi,g}$ of (\ref{wave_eq_ad})
  uniformly satisfy
  \begin{equation}
    \label{stable_obs}
    \norm{\phi}_{L^2(\Omega)} \le C \norm{\p_\nu w^\phi}_{L^2((0,T) \times \Gamma)}.
  \end{equation}
\end{definition}
A complete characterization of metrics exhibiting stable observability
is not presently known, however, it is known that stable observability holds under suitable convexity conditions. 
Indeed, if $(M,g)$ admits a strictly convex function
$\ell$ without critical points, and satisfies
\begin{equation*}
  \{x \in \p M;\ (\nabla \ell(x), \nu)_g \ge 0 \} \subset \Gamma,
\end{equation*}
then there is a neighbourhood $\mathcal G$ of $g$ and $T > 0$ such
that (\ref{wave_eq_ad}) is stably observable for $M$ and $\mathcal G$
from $\Gamma$ in time $T > 0$, see \cite{Liu2016}. 
Note that this result gives stable observability over
the complete manifold $M$ but we will need it only over the set $\Omega$.

Stable observability in the case of 
Neumann boundary condition is poorly understood presently.
For instance, stable observability can not be easily derived 
from an estimate like \cite[Th 3]{Oksanen2014},
the reason being that the 
$H^1$-norm of the Dirichlet trace of a solution to the wave equation 
is not bounded by the $L^2$-norm of the Neumann trace,
while the opposite is true \cite[Th. 4]{Oksanen2014}. See also \cite{Tataru1998} for a detailed discussion. 
For this reason we restrict our attention to the case of Dirichlet boundary condition. 

We use the notation
\begin{equation*}
W_g f(x) = u^f(T,\cdot)|_{\Omega}, \quad f \in L^2([0,T] \times \Gamma).
\end{equation*}
The stable observability (\ref{stable_obs}) says that $W_g^*$ is injective, and by duality, it implies that $W_g : L^2([0,T]
\times \Gamma) \rightarrow L^2(\Omega)$ is surjective (see 
\cite{Bardos1996}). In this case (\ref{wave_eq_aniso}) is said to exactly controllable on $\Omega$, and in particular,
for any $\phi \in L^2(\Omega)$ the control 
problem $W_g f = \phi$ has the minimum norm solution $f = W_g^\dag \phi$ given by the pseudoinverse of $W_g$.

%%%%%%%%%%%%%%%%%%%%%%%%%%%%%%%%%%%%%%%%%%%%%%%%%%%%%%%%%%%%%%%%%%%%%%%
%% %% OLD: Proof including mu's
%%%%%%%%%%%%%%%%%%%%%%%%%%%%%%%%%%%%%%%%%%%%%%%%%%%%%%%%%%%%%%%%%%%%%%%
%% \begin{lemma}
%% In any local coordinates $(x^1, \dots, x^n)$,
%% \begin{equation*}
%%   g^{jk}(x) = \frac{1}{2}\left(\Delta_{g,\tildemu}(x^l x^k) - x^k \Delta_{g,\tildemu}x^l  - x^l \Delta_{g,\tildemu}x^k\right).
%% \end{equation*}
%% \end{lemma}
%% %%\begin{proof}
%% {\em Proof.}
%% Let $(x^1,\ldots,x^n)$ be local coordinates on $M$. Define $\alpha :=
%% \tildemu\sqrt{g}$. See,
%% \begin{align*}
%%   \Delta_{g,\tildemu}(x^l x^k) &= \frac{1}{\alpha} \p_i\left(\alpha g^{ij}  \p_j (x^l x^k)\right)  \\
%%   &= \frac{1}{\alpha} \left(\p_i\left(\alpha g^{il}  x^k\right)  + \p_i\left(\alpha g^{ik}  x^l\right) \right) \\
%%   &= g^{kl}  + x^k \frac{1}{\alpha} \p_i\left(\alpha g^{il}\right)  + g^{lk}  + x^l \frac{1}{\alpha} \p_i\left(\alpha g^{ik}\right)  \\
%%   &= 2 g^{lk}  + x^k \Delta_{g,\tildemu}x^l  + x^l \Delta_{g,\tildemu}x^k. \qquad \endproof
%% \end{align*}
%% %%\end{proof}

\begin{theorem}
  \label{th_main}
Consider a family $\mathcal G$  of metrics $\tilde g$
satisfying (\ref{Omega_cond}) and suppose that
(\ref{wave_eq_ad}) is stably observable for $\Omega$ and $\mathcal G$ from $\Gamma$ in time $T > 0$.
  Let $g \in \mathcal G$ and suppose that 
  \begin{equation}
    \label{smallness_C1}
    \norm{\tilde g - g}_{C^2(M)} \le \epsilon, \quad \tilde g \in \mathcal G.
  \end{equation}
Then for small enough $\epsilon > 0$, there is $C > 0$ such that 
  \begin{equation*}
    \norm{\Psi^* \tilde g - g}_{H^{-2}(\Omega)} \le C \norm{L_{\tilde g} - L_g}_*,
    \quad \tilde g \in \mathcal G,
  \end{equation*}
  where $\Psi^* = (\Phi_g^*)^{-1} \Phi_{\tilde g}^*$ and
  \begin{equation*}
    \norm{L_g}_* = \norm{L_g}_{L^2((0,T) \times \Gamma) \to L^2(\Gamma\times(0,\epsilon))}
    + \norm{L_g \circ \p_t^2}_{L^2((0,T) \times \Gamma) \to H^{-2}(\Gamma\times(0,\epsilon))}.
  \end{equation*}
\end{theorem}
\begin{proof}
We use again the notation $\Sigma = \Gamma \times (0,r_0)$
and write also $\Sigma_T = \Gamma \times (0,T)$. Let $p \in
\Sigma$, and denote by $(x^1, \dots, x^n)$ the coordinates on
$\Sigma$ corresponding to local semi-geodesic coordinates
$(y,r)$.  Let $j,k = 1,\dots, n$ and let $\omega \subset
\Sigma$ be a neighbourhood of $p$.  Choose $\phi_\ell \in
C_0^\infty(\Sigma)$, $\ell = 1,2,3$, 
as in (\ref{def_phi_ell}).
Note that solving (\ref{scheme_anisotropic}) and taking the limit $\alpha \to 0$ is equivalent with computing $L^\dagger \phi^\ell$
see e.g. \cite[Th. 5.2]{Engl1996}.

Analogously to (\ref{scheme_anisotropic_step2}), writing the change to local coordinates explicitly,
it holds that 
\begin{eqnarray*}
(\Phi_g^* g)^{jk}(x) = 
\frac 1 2 (L_g \p_t^2 h_1(x) - &x^k L_g \p_t^2 h_2(x) - x^j L_g \p_t^2 h_3(x)),
\end{eqnarray*}
where  $h_\ell = L_g^\dagger \phi_\ell$, $\ell = 1, 2, 3$.
It will be enough to bound 
\begin{equation*}
\norm{L_{\tilde g} \p_t^2 L_{\tilde g}^\dagger \phi_\ell - L_g \p_t^2 L_g^\dagger \phi_\ell}_{H^{-2}(\omega)}, \quad \ell = 1, 2, 3,
\end{equation*}
in terms of the difference $L_{\tilde g} - L_g$.
%Let $\ell=1,2$ and use the shorthand notation $\phi = \phi_\ell$.
We have 
\begin{align*}
&\norm{L_{\tilde g} \p_t^2 L_{\tilde g}^\dagger \phi_\ell - L_g \p_t^2 L_g^\dagger \phi_\ell}_{H^{-2}(\omega)}
\\&\quad\le 
\norm{L_{\tilde g} \p_t^2 L_{\tilde g}^\dagger \phi_\ell - L_g \p_t^2 L_{\tilde g}^\dagger \phi_\ell}_{H^{-2}(\omega)}
+
\norm{L_{g} \p_t^2 L_{\tilde g}^\dagger \phi_\ell - L_g \p_t^2 L_g^\dagger \phi_\ell}_{H^{-2}(\omega)}
\\&\quad\le 
\norm{(L_{\tilde g} - L_g)  \circ \p_t^2}_{L^2(\Sigma_T) \to H^{-2}(\Sigma)}
\norm{L_{\tilde g}^\dagger}_{L^2(\Sigma) \to L^2(\Sigma_T)} \norm{\phi_\ell}_{L^2(\Sigma)}
\\&\qquad\quad+ \norm{L_g  \circ \p_t^2}_{L^2(\Sigma_T) \to H^{-2}(\Sigma)}
\norm{L_{\tilde g}^\dagger - L_g^\dagger}_{L^2(\Sigma) \to L^2(\Sigma_T)}\norm{\phi_\ell}_{L^2(\Sigma)}.
\end{align*}

We omit writing subscripts in operator norms below as 
their meaning should be clear from the context.
Pseudoinversion is continuous in the sense that
\begin{equation*}
  \norm{L_{\tilde g}^\dagger - L_g^\dagger}
  \le 3 \max\left(\norm{L_{\tilde g}^\dagger}, \norm{L_{g}^\dagger}\right)
  \norm{L_{\tilde g} - L_g},
\end{equation*}
see e.g. \cite{Izumino1983}.
It remains to show that $\norm{L_{\tilde g}^\dagger}$ is uniformly bounded for $\tilde g$ satisfying (\ref{smallness_C1}).
Note that $L_{\tilde g} = \Phi_{\tilde g}^* W_{\tilde g}$ and recall that (\ref{stable_obs}) implies
$\norm{(W_{\tilde g}^*)^\dagger} \le C$, which again implies that
$\norm{W_{\tilde g}^\dagger} \le C$.  Here the constant $C$ is uniform for $\tilde g
\in \mathcal G$.  Moreover Lemma \ref{lem_Phi_inv} implies that, for small enough $\epsilon > 0$ in
(\ref{smallness_C1}), we have $\norm{(\Phi_{\tilde g}^*)^{-1}} \le C$.
To summarize, there is uniform constant $C$ for $\tilde g$ 
satisfying (\ref{smallness_C1}) such that 
\begin{equation*}
  \norm{\Phi_{\tilde g}^* \tilde g - \Phi_g^* g}_{H^{-2}(\omega)}
  \le C \norm{L_{\tilde g} - L_g}_* \norm{\phi_\ell}_{L^2(\Sigma)}.
\end{equation*}
The claim follows by using a partition of unity. Note that the functions $\phi_\ell$
can be chosen so that they are uniformly bounded in $L^2$ when $\omega$ is varied.\qquad
%\textbf{(TODO: Add a bit more detail about partition of unity.)}
\end{proof}

\section{Computational experiment}

In this section, we provide a computational experiment to demonstrate
our approach to recovering an isotropic wave speed from the N-to-D
map. We conduct our computational experiment in the case where $M$ is
a domain in $\R^2$, however, we stress that our approach generalizes
to any $n \geq 2$.

\subsection{Forward modelling and control solutions}

For our computational experiment, we consider waves propagating in the
lower half-space $M = \R \times (-\infty,0]$ with respect to the
    following wave speed:
\begin{equation}
%% Copied from project: 2017_04m_02d_151201
%% c(x) = 1.0+0.5*x[1]-0.5*exp(-4.0*(x[0]*x[0]) - 4.0*((x[1]-0.375)*(x[1]-0.375)))
  c(x_1,x_2) = 1 + \frac{1}{2}x_2-\frac{1}{2}\exp\left(-4\left(x_1^2 + (x_2-0.375)^2\right)\right).
\end{equation}
See Figure \ref{fig:lens_model}. Waves are simulated and recorded at
the boundary for time $2T$, where $T = 1.0$. Sources are placed inside
the accessible set $\Gamma = [-\ell_s,\ell_s] \times \{0\}$, where
$\ell_s = 3.0$, and receiver measurements are made in the set $\Rec =
[-\ell_r,\ell_r] \times \{0\}$, where $\ell_r = 4.5$.

\begin{figure}
  \centering 

  \adjustbox{minipage=1.3em,valign=c}{\subcaption{}\label{fig:trip_model}}
  \begin{subfigure}[t]{\dimexpr .85\linewidth-1.3em\relax}
    \centering
    \includegraphics[width=\linewidth, valign=c]{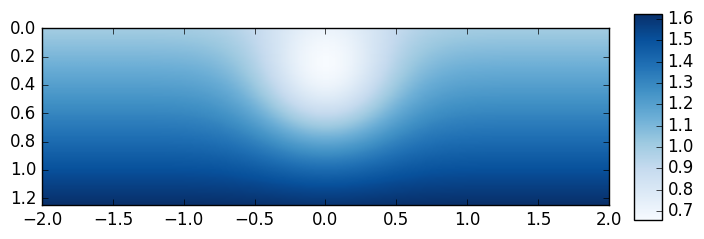}     
  \end{subfigure}
  \hfill

  \adjustbox{minipage=1.3em,valign=c}{\subcaption{}\label{fig:trip_bnc}}
  \begin{subfigure}[t]{\dimexpr .85\linewidth-1.3em\relax}
    \centering
    \includegraphics[width=\linewidth, valign=c]{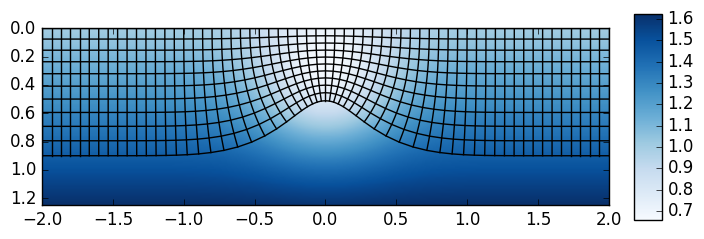}     
  \end{subfigure}
  \hfill

  \adjustbox{minipage=1.3em,valign=c}{\subcaption{}\label{fig:trip_replica}}
  \begin{subfigure}[t]{\dimexpr .85\linewidth-1.3em\relax}
    \centering
    \includegraphics[width=\linewidth, valign=c]{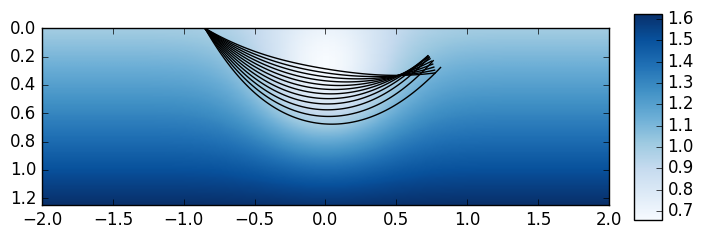}     
  \end{subfigure}

  \caption{ \label{fig:lens_model} (a) True wave speed $c$. (b)
    Semi-Geodesic coordinate grid associated with $c$. \newline (c) Some example ray
    paths with non-orthogonal intersection to $\p M$.}

\end{figure}

For sources, we use a collection of Gaussian functions spanning a
subspace of $L^2([0,T] \times \Gamma)$. Specifically, we consider
sources of the form
\begin{equation*}  
  \varphi_{i,j}(t,x) = C \exp\left(-a ((t-t_{s,i})^2 + (x-x_{s,j})^2)\right).
\end{equation*}
Here, the pairs $(t_{s,i}, x_{s,j})$ are chosen to form a uniformly
spaced grid in $[0.025,0.975] \times [-\ell_s,\ell_s]$ with spacing
$\Delta t_s = \Delta x_s = 0.025$. In total, we consider $N_{t,s} =
39$ source times $t_{s,i}$ and $N_{x,s} = 241$ source locations
$x_{s,j}$. The constant $a$, controlling the width of the basis
functions in space and time, is taken as $a = 1381.6$, and the
constant $C$ is chosen to normalize the functions $\varphi_{i,j}$ in
$L^2([0,T]\times\Gamma)$.

Wave propagation is simulated using a continuous Galerkin finite
element method with Newmark time-stepping. Waves are simulated for $t
\in [-t_0, 2T]$, where $t_0 = 0.1$, although N-to-D measurements are
only recorded in $[0,2T]$. The short buffer interval, $[-t_0, 0.0]$,
is added to the simulation interval in order to avoid numerical
dispersion from non-vanishing sources at $t = 0$. The sources are
extended to Receiver measurements are simulated by recording the
Dirichlet trace $\Lambda^{2T}_{\Gamma,\Rec} \varphi_{i,j}$ at
uniformly spaced points $x_{s,r} \in [\ell_r,\ell_r]$ with spatial
separation $\Delta x_r = 0.0125$ at uniformly spaced times $t_{s,r}
\in [0,2T]$ with temporal spacing $\Delta t_r = 0.0025$. Note that our
receiver measurements are sampled more densely in both space and time
than our source applications. In particular, $\Delta x_r = 0.5 \Delta
x_s$ and $\Delta t_r = 0.1 \Delta t_s$. In total, we take $N_{t,r} =
801$ receiver measurements at each of the $N_{x,r} = 721$ receiver
positions.

We briefly comment on the physical scales associated with the
computational experiment. In the units above, the wave speed is
approximately $1$ at the surface. If we take this to represent a wave
speed of approximately $2000$m/s and suppose that the receiver spacing
corresponds to $\Delta x_r = 12.5$m, then in the same units $\Delta
t_r = .00125$s. In addition, we have that $\ell_s = 4.5$km and $T =
1.0$s, which implies that receivers are placed within a $9.0$km region
and traces are recorded for a total of $2.0$s.  In
Fig. \ref{fig:spectrum} we plot the power spectrum for one of the
sources at a fixed source location, to give a sense of the frequencies
involved. Note that the source mostly consists of frequencies below
$15$Hz.

\begin{figure}
  \centering
  \includegraphics[width=.5\linewidth]{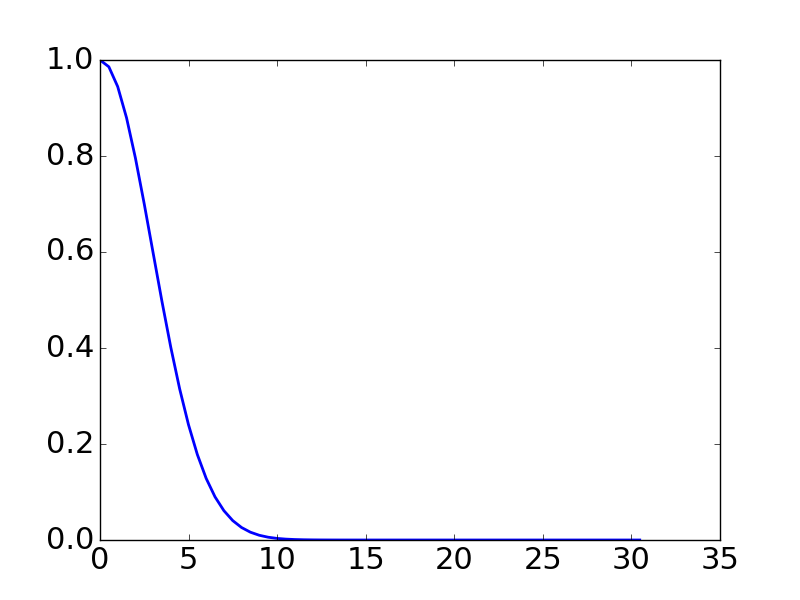}
  \caption{ \label{fig:spectrum} Power spectrum of
    $\varphi_{ij}(\cdot,x_{s,i})$, measured in Hz. We have rescaled the
    power spectrum so that it has a maximum value of $1$.}
\end{figure}

In this computational experiment, we have used sources that have a
significant frequency component at $0$~Hz. Such low frequency
contributions are not representative of physical source wavelets, so
it may be of interest to note that the data we have used can be
synthesized from sources which lack $0$~Hz components.  In particular,
these data used can be synthesized by post-processing data from
sources that are products of Gaussians in space and Ricker wavelets in
time. We note that Ricker wavelets are the second derivatives of
Gaussian functions, and that they are zero-mean sources (hence they
vanish at $0$~Hz) that are frequently used as sources when simulating
synthetic seismic data.  To demonstrate the claim, we first show that
$u^{If} = I(u^f)$, where $I$ denotes the integral $Ih(t,\cdot) :=
\int_{0}^t h(s,\cdot)\,ds$. To see this, we first observe that,
\begin{align*}
  \p_t^2(Iu^f) &= \p_t^2 \left( \int_0^t u^f(s,\cdot)\,ds\right) = \p_t u^f(t,\cdot) = \int_0^t \p_t^2 u^f(s,\cdot)\,ds - \p_t u^f(0, \cdot)\\
  &= \int_0^t c^2(x) \Delta u^f(s,\cdot) \, ds = c^2(x) \Delta (Iu^f).
\end{align*}
Here, we have used the fact that $\p_t u^f(0,\cdot) = 0$ and $(\p_t^2
- c^2(x)\Delta) u^f = 0$ since $u^f$ solves
(\ref{wave_eq_iso}). Likewise, because $u^f$ solves
(\ref{wave_eq_iso}), it follows that $\p_n (Iu^f) = I(\p_n u^f) = If$
and that $\p_t Iu^f(0,\cdot) = u^f(0,\cdot) = 0$.  Since
$Iu^f(0,\cdot) = \int_0^0 u^f(s,\cdot) \,ds = 0$, we see that $Iu^f$
satisfies:
\begin{equation*}
  \begin{array}{rcl}
    \p_t^2 w - c^2(x)\Delta w &=& 0, \quad \textnormal{in $(0,\infty) \times M$}, \\
    \p_{\vec n} w|_{x \in \p M} &=& If,  \\
    w|_{t=0}  = \p_t w|_{t=0}, &=& 0,
  \end{array}
\end{equation*}
thus $Iu^f$ solves (\ref{wave_eq_iso}) with Neumann source $If$. Since
solutions to (\ref{wave_eq_iso}) are unique, we see that $Iu^f =
u^{If}$, as claimed. An immediate consequence is that,
\begin{equation*}
  I\Lambda_{\Gamma,\Rec}^{2T}f = Iu^f|_{\Rec} = u^{If}|_\Rec = \Lambda_{\Gamma,\Rec}^{2T} I f.
\end{equation*}
Thus, $\Lambda_{\Gamma,\Rec}^{2T}I^{j} f = I^{j} \Lambda_{\Gamma,\Rec}^{2T} f$ for $j \in
\N$. Next, we let $\psi_{i,j} = \p_t^2 \varphi_{i,j}$, and note that
$\psi_{i,j}$ is a product of a Ricker wavelet in time (since it is the
second time derivative of a Gaussian function) and a Gaussian in
space. We then observe that,
\begin{equation*}
  \varphi_{i,j}(t,x) = \varphi_{i,j}(0,x) + t \p_t \varphi_{i,j}(0,x) + I^2
\psi_{i,j}(0,x).
\end{equation*} 
Under the parameter choices for $\varphi_{i,j}$, the first two terms
are considerably smaller than the third for $i \geq 4$, since $0$
belongs to the tail of the Guassian $\varphi_{ij}$. For $i \leq 3$,
the same comment holds if we replace $t = 0$ by the buffer interval
start-time, $t = -t_0$ (likewise, we would need to replace $t = 0$ by
$t = -t_0$ when applying $I$). In either event, $\varphi_{i,j} \approx
I^2 \psi_{i,j}$, and $\Lambda_{\Gamma,\Rec}^{2T}\varphi_{i,j} \approx
I^2 (\Lambda_{\Gamma,\Rec}^{2T} \psi_{i,j})$. For our particular
set-up, the N-to-D data agreed to within an error of about $1$ part in
$10^{-4}$. To recapitulate, the data that we have used could be
approximately synthesized by first using the (more) realistic sources
$\psi_{i,j}$ to simulate the data
$\Lambda_{\Gamma,\Rec}^{2T}\psi_{i,j}$, and then post-processing these
data by integrating them twice in time.

We introduce some notation, which we will use when discussing our
discretization of the connecting operator and control problems. First,
let $f \in L^2([0,T]\times\Gamma)$. We use the notation $[f]$ to
denote the vector of inner-products with entries $[f]_i = \langle f,
\varphi_i \rangle_{L^2([0,T]\times\Gamma)}$. In addition, we let
$\hat{f}$ denote the coefficient-vector for the projection of $f$ onto
$\linearSpan\{\varphi_i\}$. Let $A$ be an operator on
$L^2([0,T]\times\Gamma)$. We will use the notation $[A]$ to denote the
matrix of inner-products $[A]_{ij} = \langle A \varphi_i,
\varphi_j\rangle_{L^2([0,T]\times\Gamma)}$.  We approximate all such
integrals by successively applying the trapezoidal rule in each
dimension.

After the N-to-D data has been generated, we use the data
$\Lambda_{\Gamma}^{2T} \varphi_{i,j}$ to discretize the connecting
operator. We accomplish this using a minor modification of the
procedure outlined in \cite{Hoop2016}.  In particular, we discretize
the connecting operator by computing a discrete approximation to
(\ref{Blago}):
\begin{equation*}
  [K] = [J\Lambda_\Gamma^{2T}] -  [R\Lambda_\Gamma^T]G^{-1}[RJ].
\end{equation*}
Here, $G^{-1}$ denotes the inverse of the Gram matrix $G_{ij}
=\langle\varphi_i,\varphi_j\rangle_{L^2([0,T]\times\Gamma)}$.

Next, we describe our implementation of Lemma
\ref{lemma:approxConstControl}. Let $y \in \Gamma$, $s \in [0,T]$ and
$h \in [0,T-s]$. To obtain the control $\psi_{\alpha,h}$ associated
with $\wavecap_{\Gamma}(y,s,h)$, we solve two discrete versions of
the boundary control problem (\ref{eqn:ControlProblem}). Specifically,
for $\tau_1 = s 1_\Gamma$ and $\tau_2 = \tau_{y}^{s+h} \vee
s1_\Gamma$, we solve the discretized control problems:
\begin{equation}
  \label{eqn:discreteControl}
  ([K_{\tau_k}] + \alpha) \hat{f} = [b_{\tau_k}].
\end{equation}
This yields coefficient vectors $\hat{f}_{\alpha,k},$ for $k = 1,2$
associated with the approximate control solutions.  Here, we use the
notation $[K_{\tau_k}]$ to denote a matrix that deviates slightly from
the definition given above. In particular, we obtain $[K_{\tau_k}]$
from $[K]$ by masking rows and columns corresponding to basis
functions $\varphi_{i,j}$ localized near $(t_{s,i}, x_{s,j}) \not\in
S_{\tau_k}$. This gives an approximation to the matrix for $K_{\tau_k}
= P_{\tau_k} K P_{\tau_k}$, which we have observed performs well for
our particular basis.  The right-hand side vector $[b_{\tau_k}]$ is a
discrete approximation to $P_{\tau_k}b$ and we obtain it by first
computing the vector of inner-products $[b]_l = \langle b,
\varphi_l\rangle_{L^2([0,T]\times\Gamma)}$, and then masking the
entries of $[b]$ using the same strategy that we use to compute
$[K_{\tau_k}]$. We solve the control problems
(\ref{eqn:discreteControl}) using Matlab's back-slash function. After
computing the solutions $f_{\alpha,k}$, we then compute
$\psi_{h,\alpha} = f_{\alpha,2} - f_{\alpha,1}$.

In the inversion step, we use the boundary data to approximate
harmonic functions in semi-geodesic coordinates in the interior of
$M$. To describe this step, let $\phi$ be a harmonic function in $M$. Fix $y,
s,$ and $h$, and let $\psi_{\alpha,h}$ denote the control constructed
as in the previous paragraph. We define:
\begin{equation}
  \label{define_Hch}
  H_{c,h}\phi(y,s) := \frac{B(\psi_{h,\alpha} , \phi)}{ B(\psi_{h,\alpha} , 1)},
\end{equation}
and we calculate the right hand side directly using (\ref{define_B}).
Note that this expression coincides with an approximation to the
leading term in the right-hand side of (\ref{eqn:accuracyOfApprox}),
so for small $h$ and $\alpha$, $H_{c,h}\phi(y,s)$ will approximate
$H_c\phi(y,s)$.  However, we recall that (\ref{eqn:accuracyOfApprox})
is only accurate to $\mathcal{O}(h^{1/2})$, and in practice we found
that (\ref{define_Hch}) tends to be closer to $H_c\phi(y,s+h/2) =
\phi(x(y,s+h/2))$. This is not unexpected, since (\ref{define_Hch})
approximates $H_c\phi(y,s)$ by approximating the average of $\phi$
over $B_h = \wavecap_\Gamma(y,s,h)$, and the point $x(y,s)$ belongs to
the topological boundary of $B_h$, whereas $x(y,s+h/2)$ belongs to the
interior of $B_h$. Consequently, we will compare $H_{c,h}\phi(y,s)$ to
$H_c\phi(y,s+h/2)$ below.

\subsection{Inverting for the wave speed}

Our approach to reconstruct the wave speed $c$ consists of two
steps. In the first step, we implement Proposition
\ref{prop:harmonicReconstr} to compute an approximation to the
coordinate transform $\Phi_c$ on a grid of points $(y_i,s_j) \in
\Gamma \times [0,T]$. The second step is to differentiate the
approximate coordinate transform in the $s$-direction and to apply
(\ref{scheme_isotropic}) to compute the wave speed at the estimated
points.

To approximate the coordinate transform $\Phi_c$, we first fix a small
wave cap height $h > 0$, which we use at every grid point. The wave cap
height controls the spatial extent of the waves
$u^{\psi_{h,\alpha}}(T,\cdot)$ in the interior of $M$. Because the
vertical resolution of our basis is controlled by the separation
between sources in time, we choose $h$ to be an integral multiple of
$\Delta t_s$, and in particular, we take $h = 2\Delta t_s$.  Likewise,
we choose the grid-points $(y_i,s_j)$ to coincide with the source
centers for a subset of our basis functions. Specifically, we take
$y_i = x_{s,i}$ and $s_j = t_{s,j}$ for the source locations $x_{s,i}
\in [-1.5,1.5]$ and times $t_{s,j} \in [0.05,0.65]$. In total, the
reconstruction grid contains $N_{x,g} = 121$ horizontal positions, and
$N_{t,g} = 27$ vertical positions. Then, for each grid point
$(y_i,s_j)$ we solve (\ref{eqn:discreteControl}) for $k = 1,2$, and
obtain the source $\psi_{i,j} = \psi_{\alpha,h}$ for the point
$(y_i,s_j)$. Since the Cartesian coordinate functions $x^1$ and $x^2$
are both harmonic, we then apply (\ref{define_Hch}) to both functions
at each grid point, and define
\begin{equation}
  \label{def_phi_ch}
  \Phi_{c,h}(y_i, s_j) := \left(H_{c,h} x^1 (y_i, s_j) , H_{c,h} x^2 (y_i, s_j)\right).
\end{equation}
This yields the desired approximate coordinate transform. We plot the
estimated coordinates in Figure \ref{fig:est_coords} and compare the
estimated transform $\Phi_{c,h}(y_i,s_j)$ to the points $\Phi(y_i, s_j
+ h/2)$ in Figure \ref{fig:coord_compare}.

\begin{figure}[!htb]
  \centering 

  \adjustbox{minipage=1.3em,valign=c}{\subcaption{}\label{fig:est_coords}}
  \begin{subfigure}[t]{\dimexpr .85\linewidth-1.3em\relax}
    \centering
    \includegraphics[width=\linewidth, valign=c]{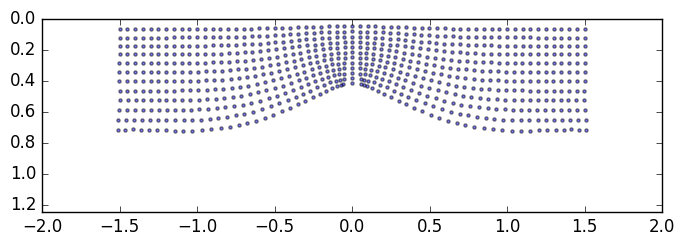}     
  \end{subfigure}
  \hfill
  
  \adjustbox{minipage=1.3em,valign=c}{\subcaption{}\label{fig:coord_compare}}
  \begin{subfigure}[t]{\dimexpr .85\linewidth-1.3em\relax}
    \centering
    \includegraphics[width=\linewidth, valign=c]{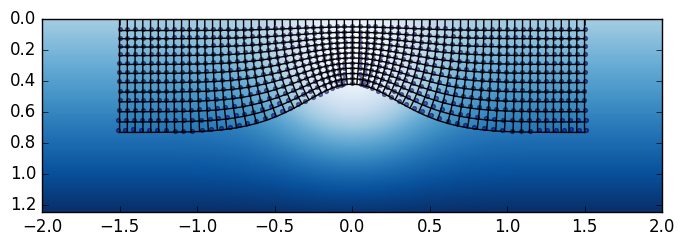}
  \end{subfigure}
  \hfill

  \caption{ (a) The estimated coordinate transform.  We have only
    plotted points for half of the $y_i$ and $s_j$. (b) Estimated
    points $\Phi_{c,h}(y_i,s_j)$ (purple dots) compared to the
    semi-geodesic coordinate grid $\Phi_c(y_i,s_j+h/2)$ (black lines)
    and wave speed.}

  \end{figure}

\begin{figure}
  \centering 

  \adjustbox{minipage=1.3em,valign=c}{\subcaption{}\label{fig:true_c}}
  \begin{subfigure}[t]{\dimexpr .85\linewidth-1.3em\relax}
    \centering
    \includegraphics[width=\linewidth, valign=c]{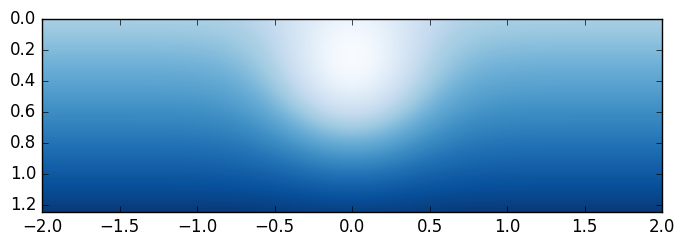}
  \end{subfigure}
  \hfill

  \adjustbox{minipage=1.3em,valign=c}{\subcaption{}\label{fig:estimated_c}}
  \begin{subfigure}[t]{\dimexpr .85\linewidth-1.3em\relax}
    \centering
    \includegraphics[width=\linewidth, valign=c]{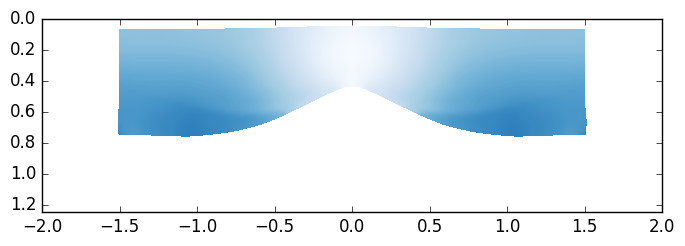}     
  \end{subfigure}
  \hfill

  \caption{ \label{fig:reconstruction_comparison} (a) True wave speed
    $c$. (b) Reconstructed wave speed, plotted at the estimated
    coordinates given by $\Phi_{c,h}$.}

\end{figure}

The last step is to approximate the wave speed. To accomplish this, we
first recall that $c(\Phi_c(y,s))^2 = |\p_s \Phi_c(y,s)|_e^2$. Thus,
for each base point $y_i$, we fit a smoothing spline to each of the
reconstructed coordinates in the $s$-direction, that is, we fit a
smoothing spline to the data sets $\{H_{c,h} x^k (y_i, s_j) : j
=1,\ldots, N_{t,g}\}$ for $k =1,2$ for each $i=1,\ldots,N_{x,g}$. We
then differentiate the resulting splines at $s_j$, for $j=1,\ldots,
N_{t,g}$ to approximate the derivatives $\p_s H_{c,h}x^k(y_i,s_j)$, at
each grid point. Finally, we estimate $c(\Phi_{c,h}(y_i,s_j))$ by
computing $|(\p_s H_{c,h}x^1(y_i,s_j),\p_s
H_{c,h}x^2(y_i,s_j))|_e$. We plot the results of this process in
Figure \ref{fig:reconstruction_comparison}, along with the true wave
speed for comparison. We also compare the reconstructed wave speed
against the true wave speed in Figure \ref{fig:compare_along_slice}
along coordinate slices.

Inspecting the bottom row of Figure \ref{fig:compare_along_slice}, we
see that the reconstruction is generally good at the estimated points.
In particular, the reconstruction quality generally decreases as $s_j$ increases,
which is expected, since the points $\Phi_{c,h}(y_i,s_j)$ with large
$s_j$ correspond to the points which are furthest from the set
$\Gamma$. Hence the N-to-D data contains a shorter window of signal
returns from these points, and thus less information about the
wave speed there.

\begin{figure}[!htb]
  \centering 
  \adjustbox{minipage=1.3em,valign=c}{}
  \begin{subfigure}[t]{\dimexpr .85\linewidth-1.3em\relax}
    \centering
    \includegraphics[width=\linewidth, valign=c]{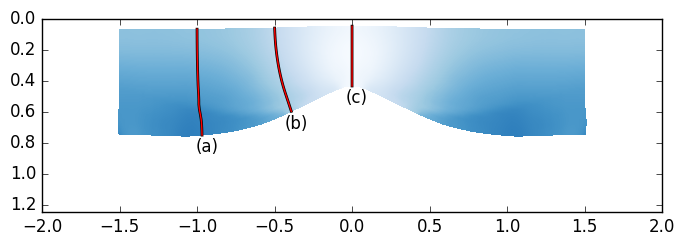}
  \end{subfigure}

  \begin{subfigure}{.32\textwidth}
    \includegraphics[width=\textwidth]{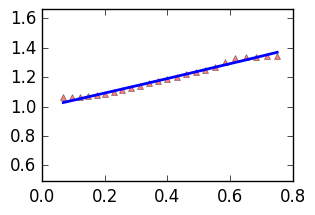}
    \subcaption{}\label{fig:test1}
  \end{subfigure}
  \begin{subfigure}{.32\textwidth}
    \includegraphics[width=\textwidth]{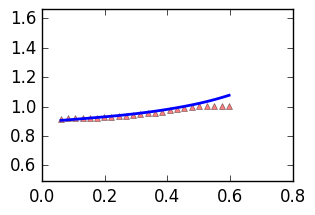}
    \subcaption{}\label{fig:test2}
  \end{subfigure}
  \begin{subfigure}{.32\textwidth}
    \includegraphics[width=\textwidth]{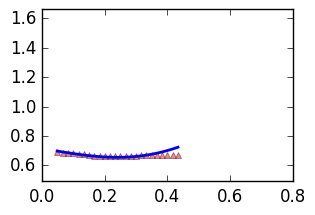}
    \subcaption{}\label{fig:test3}
  \end{subfigure}

  \caption{\label{fig:compare_along_slice} Top: Reconstructed wave
    speed with three approximated geodesics. Bottom row: true wave
    speed (blue curve) and reconstructed wave speed (red triangles)
    evaluated at the estimated coordinates for each of the indicated
    geodesics. The $x$-axis denotes the $x^2$-coordinate (depth) along
    the approximated geodesic.}
\end{figure}

\bibliographystyle{siamplain} 
\bibliography{Bibliography}

\end{document}